\documentclass{amsart}
\usepackage{amsmath,amssymb}
\usepackage{xypic} 

\newcommand{\Sym}{\mathfrak{S}}
\newcommand{\isom}{\xrightarrow{\sim}}
\newcommand{\inv}{^{-1}}

\newcommand{\dia}{\diamond}
\newcommand{\Cay}{\mathbb{O}}
\newcommand{\gmu}{\boldsymbol{\mu}}
\newcommand{\op}{\text{\textrm op}}
\newcommand{\gZ}{\mathbf{Z}}

\DeclareMathOperator{\Srd}{Srd}
\DeclareMathOperator{\charac}{char}
\DeclareMathOperator{\gGm}{\mathbf{G}_m}
\DeclareMathOperator{\gGO}{\mathbf{GO}}
\DeclareMathOperator{\gPGL}{\mathbf{PGL}}
\DeclareMathOperator{\gPGO}{\mathbf{PGO}}
\DeclareMathOperator{\gPGU}{\mathbf{PGU}}
\DeclareMathOperator{\gAut}{\mathbf{Aut}}
\DeclareMathOperator{\gInt}{\mathbf{Int}}

\DeclareMathOperator{\gH}{\mathbf{H}}
\DeclareMathOperator{\gSpin}{\mathbf{Spin}}

\DeclareMathOperator{\Nrd}{Nrd}

\DeclareMathOperator{\Orth}{O}

\DeclareMathOperator{\Spin}{Spin}

\DeclareMathOperator{\Id}{Id}

\DeclareMathOperator{\Int}{Int}

\DeclareMathOperator{\om}{\omega}
 
\DeclareMathOperator{\End}{End}
\DeclareMathOperator{\ad}{ad}
\DeclareMathOperator{\Ker}{Ker} 

\DeclareMathOperator{\SSym}{Sym} 
\DeclareMathOperator{\Aut}{Aut} 
\DeclareMathOperator{\Trd}{Trd} 
\DeclareMathOperator{\GO}{GO}
\DeclareMathOperator{\PGO}{PGO}

\theoremstyle{plain} %
\numberwithin{equation}{section}

\newtheorem{theorem}[equation]{Theorem}
\newtheorem{thm}[equation]{Theorem}

\newtheorem{prop}[equation]{Proposition}
\newtheorem{proposition}[equation]{Proposition}

\newtheorem{cor}[equation]{Corollary}
\newtheorem{lemma}[equation]{Lemma}
\theoremstyle{definition} %
\numberwithin{equation}{section}
\newtheorem{example}[equation]{Example}

\newtheorem{defn}[equation]{Definition}
\newtheorem{defns}[equation]{Definitions}

\dedicatory{Dedicated with great friendship to Eva Bayer on the
  occasion of her $60^{th}$ birthday}

\title[Conjugacy classes of trialitarian automorphisms]{Conjugacy
  classes of trialitarian automorphisms and symmetric compositions} 
\author{V. Chernousov \and M.-A. Knus \and J.-P. Tignol}
\date{June 25, 2011}
\address{Department of Mathematical and Statistical Sciences\\
University of Alberta\\
632 Central Academic Building\\
Edmonton, AB  T6G 2G1, Canada}
\email{chernous@math.ualberta.ca}
\address{Department Mathematik\\
  ETH Zentrum\\
  CH-8092 Z\"urich\\
  Switzerland} \email{knus@math.ethz.ch}
\address{ICTEAM Institute,
  Universit\'e catholique de Louvain\\
  B-1348 Louvain-la-Neuve, Belgium\\
  and Zukunftskolleg, Universit\"at Konstanz\\
  D-78457 Konstanz, Germany} \email{jean-pierre.tignol@uclouvain.be}
\thanks{The first author 
  was partially supported by
the Canada Research Chairs Program and an~NSERC research grant, the third author
 by the~F.R.S.--FNRS (Belgium)}


\begin{document}
\maketitle

\begin{abstract}
  The trialitarian automorphisms considered in this paper are the
  outer automorphisms of order~3 of adjoint classical groups of
  type~$\mathrm{D}_4$ over arbitrary fields. A one-to-one
  correspondence is established 
  between their conjugacy classes and similarity classes of symmetric
  compositions on $8$-dimensional quadratic spaces. Using the known
  classification of symmetric compositions, we distinguish two
  conjugacy classes of trialitarian automorphisms over algebraically
  closed fields. For type~I, the group of fixed points is of
  type~$\mathrm{G}_2$, whereas it is of type~$\mathrm{A}_2$ for
  trialitarian automorphisms of type~II.
\end{abstract}

\section{Introduction}
Among simple algebraic groups of classical type only the simple
adjoint algebraic groups $G = \gPGO^+(n)$ and the simple simply
connected algebraic groups $G = \gSpin(n)$, where $n$ is the norm of
an octonion algebra, admit outer automorphisms of order $3$, known as
trialitarian automorphisms, see \cite[(42.7)]{KMRT} or \cite{jactri}.
These groups are of type $\mathrm{D}_4$ and there is a split exact sequence of
algebraic groups
\begin{equation}
\label{eq:exactseq}
  1 \to \gInt(G) \to \gAut(G)
   \to \Sym_3 \to 1
\end{equation}
where the permutation group of three elements $\Sym_3$ is viewed as
the group of automorphisms of the Dynkin diagram of type $\mathrm{D}_4$:
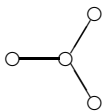
\[
\begin{picture}(100,50)
\put(40,25){\circle{5}}
\put(35,15){}
\put(42.5,25){\line(1,0){15}}
\put(60,25){\circle{5}}
\put(52,15){}
\put(71,42){\circle{5}}
\put(76,41){}
\put(71,8){\circle{5}}
\put(76,5){}
\put(62,27){\line(3,5){7.5}}
\put(62,23){\line(3,-5){7.5}}
\end{picture}
\]

In view of the exact sequence \eqref{eq:exactseq}, all the
trialitarian automorphisms of $G$ can be obtained from a fixed one
$\rho$ by composing $\rho$ or $\rho^{-1}$ with inner
automorphisms. They are not necessarily conjugate to $\rho$,
however. Our goal is to classify trialitarian automorphisms defined
over an arbitrary field $F$ up to conjugation in the group
$\gAut(G)(F)$ of $F$-automorphisms of $G$. We achieve this goal by
relating trialitarian automorphisms to symmetric compositions and
using the known classification of composition algebras.

We take as guiding principle the analogous description of outer
automorphisms of order~$2$ of $\gPGL_n$ for $n\geq3$, which could be
termed \emph{dualitarian automorphisms}: they have the form $f\mapsto
\sigma(f)^{-1}$, where $\sigma$ is the adjoint involution of a
nonsingular symmetric or skew-symmetric bilinear form. Dualitarian
automorphisms on $\gPGL_n$ are 
thus in one-to-one correspondence with nonsingular symmetric or
skew-symmetric bilinear forms on 
an $n$-dimensional vector space up to scalar multiples, and
dualitarian automorphisms are conjugate if and only if the
corresponding bilinear forms are similar. There are two types of
dualitarian automorphisms, distinguished by the type of their groups
of fixed points, which can be either symplectic or (in characteristic
different from~$2$) orthogonal. In characteristic~$2$, the
non-symplectic case leads to group schemes that are not smooth.

Likewise, we set up a one-to-one correspondence between trialitarian
automorphisms of $\gPGO^+(n)$, for $n$ a $3$-fold Pfister quadratic
form, and symmetric compositions up to scalar multiples on the
underlying vector space of $n$, and use it to define a bijection
between conjugacy classes of trialitarian automorphisms and similarity
classes of symmetric compositions. We isolate two types of symmetric
compositions. Type~I is related to octonion algebras; the fixed
subgroups are of type~$\mathrm{G}_2$. When the characteristic is not~$3$,
type~II is related to central simple algebras of degree~$3$; the fixed
subgroups are of type~${}^1\!\mathrm{A}_2$ or ${}^2\!\mathrm{A}_2$. In
characteristic~$3$, the fixed subgroups under trialitarian
automorphisms of type~II are not smooth.

Triality for simple Lie groups first appears in the paper
\cite{Cartan25} of  \'E. Cartan, who already noticed that octonions
can be used to explicitly define trialitarian automorphisms.\footnote{  
We refer to \cite[pp.~510--511]{KMRT} and
\cite[\S3.8]{springerveldkamp} for historical comments on 
triality.} The observation that symmetric compositions are
particularly well suited for that purpose is due to M.~Rost.

As far as we know a complete classification of trialitarian
automorphisms  of simple groups of type $\mathrm{D}_4$ had only been obtained
over finite fields, in \cite[(9.1)]{gorensteinlyons:83}, and in 
\cite[Theorem~(4.7.1)]{gorensteinlyonssolomon:98} for the algebraic
closure of a finite field of characteristic different from $3$. A summary of
known  results for Lie algebras can be found in \cite{knus09}. 
Like for duality, there is also a projective geometric version of triality,
which is in fact older than Cartan's triality and goes back to
Study \cite{study:13}. A classification of geometric trialities, as well as
of  their groups of automorphisms,
was done by Tits  in \cite{Tits59}.

\medskip

If not explicitly mentioned $F$ denotes throughout the paper an arbitrary field.

\section{Similarities of quadratic spaces}

A quadratic space over $F$ is a finite-dimensional vector space $V$
over $F$ with a quadratic form $q\colon V\to F$. We always assume that
$q$ is nonsingular, in the sense that the \emph{polar bilinear form} $b_q$
defined by
\[
b_q(x,y)=q(x+y)-q(x)-q(y)\qquad\text{for $x$, $y\in V$}
\]
has radical $\{0\}$. We also assume throughout that $\dim V$ is \emph{even}.
Let $\ad_q$ denote the
involution on $\End_FV$ such that
\[
b_q\big(f (x), y ) = b_q(x, \ad_{q}(f)(y)\big)\qquad\text{for all
  $f\in\End_FV$ and $x$, $y\in V$.}
\]
This involution is said to be \emph{adjoint} to $q$. Let $\gGO(q)$ be
the $F$-algebraic group of similarities of $(V,q)$, whose group of
rational points $\gGO(q)(F)$ consists of linear maps $f\colon V\to V$
for which there exists a scalar $\mu(f)\in F^\times$, called the
\emph{multiplier} of $f$, such that
\[
q\bigl(f(x)\bigr)=\mu(f)q(x)\qquad\text{for all $x\in V$.}
\]
The center of $\gGO(q)$ is the multiplicative group $\gGm$,
whose rational points are viewed as homotheties. Let $\gPGO(q)$ be the
$F$-algebraic group of automorphisms of $\End_FV$ that commute with
the adjoint involution $\ad_q$. This group is identified with the
quotient $\gGO(q)/\gGm$, acting on $\End_FV$ by inner automorphisms:
for $f\in\gGO(q)(F)$, we let $[f]$ be the image of $f$ in
$\gGO(q)(F)/\gGm(F)$ and identify $[f]$ with
\[
\Int[f]\colon \End_F V\to \End_F V, \qquad \phi\mapsto f\phi f^{-1},
\]
see \cite[\S23]{KMRT}. For simplicity, we write
\[
\GO(q)=\gGO(q)(F)\qquad\text{and}\qquad
\PGO(q)=\gPGO(q)(F)=\GO(q)/F^\times.
\]

Let $C(V,q)$ be the Clifford algebra of the quadratic space $(V,q)$ and let 
$C_0(V,q)$ be the even Clifford algebra.  We let $\sigma$ be the
canonical involution of $C(V,q)$, such that $\sigma(x)=x$ for $x\in
V$, and use the same notation for its restriction to $C_0(V,q)$. 
Every similarity $f\in\GO(q)$ induces an automorphism $C_0(f)$ of
$(C_0(V,q),\sigma)$ such that 
\begin{equation} 
  \label{equ:simactionCo}
  C_0(f)(xy)=\mu(f)^{-1}f(x)f(y)\qquad\text{for $x$, $y\in V$,}
\end{equation}
see \cite[(13.1)]{KMRT}.
This automorphism depends only on the image $[f]=fF^\times$ of $f$ in
$\PGO(q)$, and we shall use the notation $C_0[f]$ for $C_0(f)$.
The similarity $f$ is \emph{proper} if
$C_0[f]$ fixes the center of $C_0(V,q)$ and 
\emph{improper} if it induces a nontrivial automorphism of
the center of $C_0(V,q)$ (see \cite[(13.2)]{KMRT}).
Proper similarities define an algebraic subgroup $\gGO^+(q)$ in
$\gGO(q)$, and we let $\gPGO^+(q)=\gGO^+(q)/\gGm$, a subgroup of
$\gPGO(q)$. The  groups $\gGO^+(q)$ and $\gPGO^+(q)$ are the connected
components of the identity in $\gGO(q)$ and $\gPGO(q)$ respectively,
see~\cite[\S23.B]{KMRT}. Conjugation by an improper similarity is an
outer automorphism of $\gPGO^+(q)$, since the induced automorphism on
the center of $C_0(V,q)$ is nontrivial. As pointed out in the
introduction, more outer automorphisms can be defined when the form
$q$ is the norm of an octonion algebra, i.e., a $3$-fold Pfister
form. In this case, we call the quadratic space a \emph{$3$-fold
  Pfister quadratic space}.

\section{Symmetric compositions}

Let $(S,n)$ be a quadratic space of dimension~$8$ over $F$.

\begin{defn} \label{def:symcomp}
A \emph{symmetric composition} on $(S,n)$ is
 an $F$-bilinear map
\[
\star\colon S\times S\to S, \qquad (x,y)\mapsto x\star y\quad\text{for
  $x$, $y\in S$}
\]
subject to the following conditions:
\begin{enumerate} \label{enu:defsym}
\item[(1)]
  there exists $\lambda_\star\in F^\times$, called the
  \emph{multiplier} of the symmetric composition~$\star$, such that
  \[
  n(x\star y)=\lambda_\star n(x) n(y)\qquad\text{for all $x$, $y\in
    S$;}
  \]
  \item[(2)]
  for all $x$, $y$, $z\in S$,
  \[
  b_n(x\star y, z) = b_n(x, y\star z).
  \]
\end{enumerate}
 A symmetric composition with multiplier $\lambda_\star =1$ is called
  \emph{normalized}.
\end{defn}

This definition of symmetric compositions  
is not identical to the one given in \cite[\S34]{KMRT}, where
$(S,n)$ can \emph{a priori} be a nonsingular quadratic space of
arbitrary finite dimension (but in fact $\dim S=1$, $2$, $4$ or $8$ by
a theorem of Hurwitz, see \cite[(33.28)]{KMRT}), and all the symmetric
compositions are normalized. 

\medskip

The following lemma is an immediate consequence of the definition of
a symmetric composition:

\begin{lemma} Let $\star$ be a symmetric composition on $(S,n)$
with multiplier $\lambda_\star $.
\begin{enumerate}
\item
For any scalar $\nu\in F^\times$,  $\star$ is  a symmetric composition on the
quadratic space $(S,\nu n)$ with
multiplier  $\nu^{-1}\lambda_\star$.
\item
For any scalar $\nu\in F^\times$, the bilinear map 
$\nu\cdot \star \colon (x,y) \mapsto \nu x\star y$
is a symmetric composition on the
quadratic space $(S,n)$ with
multiplier  $\nu^{2}\lambda_\star$.
\end{enumerate}
\end{lemma}

\begin{lemma} 
  \label{lem:2.1} 
  Under condition (1), condition (2) of the definition of a symmetric
  composition is equivalent to
  \[
  x\star(y\star x) = \lambda_\star n(x)y=(x\star y)\star x
  \qquad\text{for all $x$, $y\in S$.}
  \]
\end{lemma}

\begin{proof} 
  The claim follows by applying  \cite[(34.1)]{KMRT} to the
  (normalized) composition $\star$ on $(S,\lambda_\star n)$.
\end{proof}

Since the symmetric compositions do not change when the quadratic form
$n$ is scaled, we may and will always assume without loss of
generality that $n$ represents~$1$. It is then a $3$-fold Pfister
form, by~\cite[(33.18), (33.29)]{KMRT}.

\begin{example}
  \label{ex:cayley}
  Let $(\Cay,n)$ be an octonion algebra with norm $n$,
  multiplication $(x,y)\mapsto x\cdot y$,  identity $1$, and conjugation
  $x\mapsto\overline{x}$. The multiplication
  \[
  x\star y= \overline{x}\cdot \overline{y}
  \]
  defines a normalized symmetric composition on $(\Cay,n)$,
  called the \emph{para-octonion composition} (see for example
  \cite[\S 34.A]{KMRT}). Observe that 
  $\overline{x \star y} =  \overline{y} \star \overline{x}$ for all
  $x,\,y \in \Cay$. 
\end{example}

More examples---and a complete classification of symmetric
compositions---are given in
Section~\ref{section:classificationsymcomp}.

\begin{defns}
  Let $(S,n)$ be a $3$-fold Pfister quadratic space and let $\star$
  and $\diamond$ be symmetric compositions on $(S,n)$. A
  \emph{similarity} $f\colon\star\to\diamond$ is an element $f\in\GO(n)$
  such that
  \[
  f(x\star y) =
  f(x)\dia f(y)\qquad\text{for all $x$, $y\in S$.}
  \]
  The multipliers of $f$, $\star$ and $\dia$ are then related by
  $\lambda_\star=\lambda_\dia \mu(f)$.  Similarities with
  multiplier~$\mu(f)=1$ are called \emph{isomorphisms}. In particular,
  similarities between symmetric compositions with the same multiplier
  are isomorphisms.

  The \emph{opposite} of the symmetric composition $\star$ on $(S,n)$
  is the symmetric composition $\star^\op$ on $(S,n)$ defined by
  \[
  x\star^\op y= y\star x\qquad\text{for $x$, $y\in S$.}
  \]
  The multiplier of $\star^\op$ is the same as the multiplier of
  $\star$.

  Symmetric compositions $\star$ and $\dia$ are said to be \emph{similar}
  if there is a similarity $\star\to\dia$; they are said to be
  \emph{antisimilar} if $\star^\op$ and $\dia$ are similar.
\end{defns}

\begin{prop} 
  \label{prop:normalize}
  Let $n$ be a $3$-fold quadratic Pfister form on a vector space
  $S$, and let $\star$ be a symmetric composition on $(S,n)$. There
  exists up to isomorphism a unique normalized composition $\dia$ on
  $(S,n)$ similar to $\star$.
\end{prop}

\begin{proof}
  Let $u \in S$ be such that $n(u)=1$. Consider the maps
  $\ell_u^\star$, $r_u^\star\colon S\to S$ defined by
  \[
  \ell_u^\star(x)=u\star x\qquad\text{and}\qquad r_u^\star(x)=x\star u
  \qquad\text{for $x\in S$.}
  \]
  Define a new multiplication $\dia$ on $S$ by
  \[
  x\dia y= \lambda_\star^{-2}\ell_u^\star\bigl(r_u^\star(x)\star
  r_u^\star(y)\bigr) \qquad\text{for $x$, $y\in S$.}
  \]
  Condition~(1) for a symmetric composition yields
  \[
  n\bigl(\ell_u^\star(x)\bigr) = \lambda_\star n(x) =
  n\bigl(r_u^\star(x)\bigr) \qquad\text{for all $x\in S$.}
  \]
  It is then easy to check that $\dia$ is a normalized symmetric
  composition on $(S,n)$. Moreover, by Lemma~\ref{lem:2.1}, we have
  \[
  \ell_u^\star\circ r_u^\star = r_u^\star\circ\ell_u^\star =
  \lambda_\star\Id_S.
  \]
  \relax From the definition of $\dia$, it follows that
  \[
  \lambda_\star^{-1}r_u^\star(x\dia y) = \bigl(\lambda_\star^{-1}
  r_u^\star(x)\bigr) \star \bigl(\lambda_\star^{-1} r_u^\star(y)\bigr)
  \qquad\text{for all $x$, $y\in S$,}
  \]
  hence $\lambda_\star^{-1}r_u^\star\colon\dia\to\star$ is a
  similarity. If $\dia'$ is another normalized composition similar to
  $\star$, then $\dia$ and $\dia'$ are similar, hence isomorphic since
  they have the same multiplier.
\end{proof}

\section{From symmetric compositions to trialitarian automorphisms}

Throughout this section, we fix a $3$-fold Pfister quadratic space
$(S,n)$ and a symmetric composition $\star$ on $(S,n)$ with multiplier
$\lambda_\star$. We show how to associate to this composition a
trialitarian automorphism $\rho_\star$ of $\gPGO^+(n)$ defined over
$F$, i.e., an outer automorphism of order~$3$ in
$\gAut\bigl(\gPGO^+(n)\bigr)(F)$.

For each element $x\in S$ we define linear maps
$\ell^\star_x$, $r^\star_x\colon S\to S$ by 
\[
 \ell_x^\star (y) = x \star y \qquad \text{and}\qquad r_x^\star(y) = y \star x \qquad 
 \text{for $y \in S$.}
 \]
Consider 
$
\begin{pmatrix}
  0&\lambda\inv _\star r_x^\star\\ \ell_x^\star&0
\end{pmatrix}\in M_2(\End S) = \End _F(S\oplus S)$. By Lemma~\ref{lem:2.1}, we have
\[
\begin{pmatrix}
  0 & \lambda\inv _\star r_x^\star\\ \ell_x^\star&0
\end{pmatrix}^2= n(x)\cdot \Id_{S\oplus S}.
\]
Therefore, by the universal property of Clifford algebras, the map

\[
x\mapsto
\begin{pmatrix}
  0&\lambda_\star\inv r_x^\star\\ \ell_x^\star&0
\end{pmatrix}, \quad \text{$x \in S$}
\]
induces an $F$-algebra homomorphism 
\[
  \alpha^\star \colon C(S,n)\to \End _F(S\oplus S).
\]

\begin{prop}
  \label{prop:triasimilituderevone}
  The map $ \alpha^\star $ is an isomorphism of $F$-algebras
  \[
  \alpha^\star \colon C(S,n) \isom \End_F(S \oplus S).
  \]
  It restricts to an isomorphism of $F$-algebras with involution
  \[
  \alpha^\star_0\colon \big(C_0(S,n),\sigma\big) \isom (\End_F
  S,\ad_{n})\times (\End_F S, \ad_{n})
  \]
  such that for $x$, $y\in S$
  \[
  \alpha^\star_0(x\cdot y) = (\lambda_\star^{-1}r_x^\star \ell_y^\star,
  \lambda_\star^{-1} \ell_x^\star r_y^\star).
  \]
\end{prop}

\begin{proof}
  This is shown in \cite[(35.1)]{KMRT}. For completeness, we reproduce
  the easy argument.  The map $\alpha^\star$ is injective since the
  Clifford algebra $C(S, n)$ is simple. It restricts to an $F$-algebra
  embedding $C_0(S, n)\hookrightarrow (\End S)\times(\End S)$. Since
  $\dim S=8$, dimension count shows that this embedding is an
  isomorphism. Using Lemma~\ref{lem:2.1}, it is easy to check that the
  involution $\sigma$ corresponds to $\ad_n\times \ad_n$ under
  $\alpha_0^\star$.
\end{proof}

For $f$, $g$, $h\in\GO(n)$, define the $F$-algebra automorphism
$\Theta(g,h)$ of $(\End_FS)\times(\End_FS)$ by
\[
\Theta(g,h)\colon
(\varphi,\psi)\mapsto(g\psi g^{-1},h\varphi h^{-1})
\]
and consider the following diagrams:
\begin{equation*}
  \tag*{$D_\star^+(f,g,h)$}
  \xymatrix{
  C_0(S,n) \ar[r]^(.35){\alpha_0^\star} \ar[d]_{C_0[f]} &
  (\End_FS)\times(\End_FS) \ar[d]^{\Int[g]\times\Int[h]}\\
  C_0(S,n) \ar[r]^(.35){\alpha_0^\star}& (\End_FS)\times(\End_FS)
  }
\end{equation*}
and
\begin{equation*}
  \tag*{$D_\star^-(f,g,h)$}
  \xymatrix{
  C_0(S,n) \ar[r]^(.35){\alpha_0^\star} \ar[d]_{C_0[f]} &
  (\End_FS)\times(\End_FS) \ar[d]^{\Theta(g,h)}\\
  C_0(S,n) \ar[r]^(.35){\alpha_0^\star}& (\End_FS)\times(\End_FS).
  }
\end{equation*}

\begin{lemma}
  \label{lem:3}
  Let $\star$ be a symmetric composition on $(S,n)$.
  For $f$, $g$, $h\in\GO(n)$,
  the following statements are equivalent:
  \begin{enumerate}
  \item[($a^+$)]
  there exists a scalar $\lambda\in F^\times$ such that
  \[
  \lambda f(x\star y) = g(x)\star h(y)\qquad\text{for all $x$, $y\in
    S$;}
  \]
  \item[($b^+$)]
  there exists a scalar $\mu\in F^\times$ such that
  \[
  \mu g(x\star y) = h(x)\star f(y) \qquad\text{for all $x$, $y\in S$;}
  \]
  \item[($c^+$)]
  there exists a scalar $\nu\in F^\times$ such that
  \[
  \nu h(x\star y) = f(x)\star g(y) \qquad\text{for all $x$, $y\in S$;}
  \]
  \item[($d^+$)]
  the diagram $D_\star^+(f,g,h)$ commutes;
  \item[($e^+$)]
  the diagram $D_\star^+(g,h,f)$ commutes;
  \item[($f^+$)]
  the diagram $D_\star^+(h,f,g)$ commutes.
  \end{enumerate}
  When they hold, the scalars $\lambda$, $\mu$, $\nu$ and the
  multipliers of $f$, $g$, $h$ are related by
  \[
  \lambda\mu=\mu(h),\quad \mu\nu=\mu(f),\quad \lambda\nu=\mu(g).
  \]
  Moreover, the similarities $f$, $g$, and $h$ are all proper in this
  case.

  Likewise, for $f$, $g$, $h\in\GO(n)$ the following statements are
  equivalent:
  \begin{enumerate}
  \item[($a^-$)]
  there exists a scalar $\lambda\in F^\times$ such that
  \[
  \lambda f(x\star y)= h(y)\star g(x)\qquad\text{for all $x$, $y\in
    S$;}
  \]
  \item[($b^-$)]
  there exists a scalar $\mu\in F^\times$ such that
  \[
  \mu g(x\star y) = f(y)\star h(x)\qquad\text{for all $x$, $y\in S$;}
  \]
  \item[($c^-$)]
  there exists a scalar $\nu\in F^\times$ such that
  \[
  \nu h(x\star y) = g(y)\star f(x)\qquad\text{for all $x$, $y\in S$;}
  \]
  \item[($d^-$)]
  the diagram $D_\star^-(f,h,g)$ commutes;
  \item[($e^-$)]
  the diagram $D_\star^-(g,f,h)$ commutes;
  \item[($f^-$)]
  the diagram $D_\star^-(h,g,f)$ commutes.
  \end{enumerate}
  When they hold, the scalars $\lambda$, $\mu$, $\nu$ and the
  multipliers of $f$, $g$, $h$ are related by
  \[
  \lambda\mu=\mu(h),\quad \mu\nu=\mu(f),\quad \lambda\mu= \mu(g).
  \]
  Moreover, the similarities $f$, $g$, and $h$ are all improper in
  this case.
\end{lemma}

\begin{proof}
  This is essentially proved in \cite[(35.4)]{KMRT}. We give a proof
  for the reader's convenience.

  ($a^+$)$\Rightarrow$($b^+$) Multiplying each side of ($a^+$) on the
  left by $h(y)$ and using Lemma~\ref{lem:2.1}, we obtain
  \[
  \lambda h(y)\star f(x\star y) = \lambda_\star n\bigl(h(y)\bigr) g(x)
  \qquad\text{for $x$, $y\in S$.}
  \]
  If $y$ is anisotropic, the map $r_y^\star$ is a bijection whose
  inverse is $\lambda_\star^{-1}n(y)^{-1}\ell_y^\star$. Letting $X=y$
  and $Y=x\star y$, we derive from the preceding equation
  \[
  \lambda h(X)\star f(Y)= \mu(h) g(X\star Y) \qquad\text{for $X$,
    $Y\in S$ with $Y$ anisotropic.}
  \]
  Since generic vectors in $S$ are anisotropic, ($b^+$) follows with
  $\mu=\mu(h)\lambda^{-1}$. Similar arguments yield
  ($b^+$)$\Rightarrow$($c^+$) 
  with $\nu=\mu(f)\mu^{-1}$ and ($c^+$)$\Rightarrow$($a^+$) with
  $\lambda=\mu(g)\nu^{-1}$.

  Now, assume ($a^+$), ($b^+$), and ($c^+$) hold. From ($b^+$) and
  ($c^+$), and from $\mu\nu=\mu(f)$, we readily derive 
  \[
  g\bigl((y\star z)\star x\bigr) = \mu^{-1} h(y\star z)\star f(x) =
  \mu(f)^{-1} \bigl(f(y)\star g(z)\bigr)\star f(x)
  \]
  and
  \[
  h\bigl(x\star(z\star y)\bigr) = \nu^{-1}f(x)\star g(z\star y) =
  \mu(f)^{-1} f(x)\star\bigl(h(z)\star f(y)\bigr)
  \]
  for all $x$, $y$, $z\in S$. These equations mean that diagram
  $D_\star^+(f,g,h)$ commutes, hence ($d^+$) holds. Similar computations
  show that ($e^+$) and ($f^+$) hold.

  Now, assume ($d^+$) holds. The map
  \[
  x\mapsto
  \begin{pmatrix}
    0&\lambda_\star^{-1}r_{f(x)}^\star\\
    \mu(f)^{-1} \ell_{f(x)}^\star&0
  \end{pmatrix}
  \]
  also yields by the universal property of Clifford algebras an
  $F$-algebra isomorphism $\beta\colon C(S,n)\isom \End_F(S\oplus
  S)$. The automorphism $\beta\circ(\alpha^\star)^{-1}$ of
  $\End_F(S\oplus S)$ is inner by the Skolem--Noether theorem, and it
  preserves $(\End_FS)\times(\End_FS)$ diagonally embedded, hence
  \[
  \beta\circ(\alpha^\star)^{-1} = \Int
  \begin{pmatrix}
    \varphi&0\\0&\psi
  \end{pmatrix}
  \qquad\text{for some invertible $\varphi$, $\psi\in\End_FS$.}
  \]
  It follows that for $x\in S$
  \[
  \varphi\circ r_x^\star\circ\psi^{-1}=r_{f(x)}^\star
  \qquad\text{and}\qquad \psi\circ\ell_x^\star\circ\varphi^{-1} =
  \mu(f)^{-1}\ell_{f(x)}^\star,
  \]
  which means that for $x$, $y\in S$
  \begin{equation}
    \label{eq:phipsi}
    \varphi(y\star x) = \psi(y)\star f(x) \qquad\text{and}\qquad
    \psi(x\star y) = \mu(f)^{-1}f(x)\star\varphi(y).
  \end{equation}
  For $x$, $y\in S$ we have
  \[
  \beta(x\cdot y) = (\lambda_\star^{-1}\mu(f)^{-1} r_{f(x)}^\star
  \ell_{f(x)}^\star, \lambda_\star^{-1}\mu(f)^{-1}\ell_{f(x)}^\star
  r_{f(x)}^\star) = \alpha_0^\star\circ C_0[f](x\cdot y),
  \]
  hence ($d^+$) yields
  \[
  \Int
  \begin{pmatrix}
    \varphi&0\\0&\psi
  \end{pmatrix}
  =
  \Int
  \begin{pmatrix}
    g&0\\0&h
  \end{pmatrix}.
  \]
  It follows that $[\varphi]=[g]$ and $[\psi]=[h]$, hence
  \eqref{eq:phipsi} yields~($c^+$) and ($b^+$). Similarly,
  ($e^+$)$\Rightarrow$($a^+$), ($c^+$), and
  ($f^+$)$\Rightarrow$($a^+$), ($b^+$). Thus, all the statements
  ($a^+$)--($f^+$) are equivalent. When they hold, ($d^+$) shows that
  $C_0[f]$ is the identity on the center of $C_0(S,n)$, hence $f$ is
  proper. Similarly, ($e^+$) shows that $g$ is proper and ($f^+$) that
  $h$ is proper.

  The proof of the second part is similar. Details are left to the reader.
\end{proof}

\begin{cor}
  \label{cor:autoproper}
  Let $\star$ be a symmetric composition on $(S,n)$. Every
  automorphism $\star\to\star$ is a proper isometry.
\end{cor}

\begin{proof}
  For any automorphism $f$, condition~($a^+$) of Lemma~\ref{lem:3}
  holds with $\lambda=1$ and $g=h=f$.
\end{proof}

\begin{thm}
  \label{thm:4}
  Let $\star$ be a symmetric composition on $(S,n)$.
  \begin{enumerate}
  \item[(1)]
  For any $f\in\GO^+(n)$, there exist $g$, $h\in\GO^+(n)$ such that
  the equivalent conditions $(a^+)$--$(f^+)$ in Lemma~\ref{lem:3}
  hold. 
  \item[(2)]
  For any improper similarity $f\in\GO(n)$, there exist improper
  similarities $g$, $h\in\GO(n)$ such that the equivalent conditions
  $(a^-)$--$(f^-)$ in Lemma~\ref{lem:3} hold.
  \end{enumerate}
  In each case, the similarities $g$ and $h$ are uniquely determined
  up to multiplication by scalars, so $[g]$ and 
  $[h]\in\PGO(n)$ are uniquely determined.
\end{thm}
  
\begin{proof}
(1)
The automorphism
$C_0[f]$ of $C_0(S,n)$ preserves the canonical involution $\sigma$ and
restricts to the identity on the center. Therefore, we may find $g$,
$h\in\GO(n)$ such that the diagram $D_\star^+(f,h,g)$ commutes. The elements
$[g]$, $[h]\in\PGO(n)$ are uniquely determined by this condition (or
by any other in the list ($a^+$)--($f^+$) from Lemma~\ref{lem:3}),
and Lemma~\ref{lem:3} shows that $g$ and $h$ are proper
similarities.

(2)
If $f$ is improper, then $C_0[f]$ restricts to the nontrivial
automorphism of the center of $C_0(S,n)$, hence it fits in a diagram
$D_\star^-(f,h,g)$ for some similarities $g$, $h\in\GO(n)$, which are
improper by Lemma~\ref{lem:3}.
\end{proof}

Note that scaling $g$ and/or $h$ we may change the scalars $\lambda$,
$\mu$, $\nu$ in conditions ($a^+$)--($c^+$) or ($a^-$)--($c^-$). We
may for instance choose $g$ and $h$ so that $\lambda=1$, or, as in
\cite[(35.4)]{KMRT}, so that $\lambda=\mu(f)^{-1}$, $\mu=\mu(g)^{-1}$,
and $\nu=\mu(h)^{-1}$. In this case, $\mu(f)\mu(g)\mu(h)=1$.

\medskip

In view of Theorem~\ref{thm:4}, we may define a map
\[
\rho_\star\colon\PGO^+(n)\to\PGO^+(n)
\]
by carrying any $[f]\in\PGO^+(n)$ to the unique $[g]\in\PGO^+(n)$ such
that the diagram $D_\star^+(f,g,h)$ commutes for some $h\in\GO^+(n)$.
Since $\rho_\star$ is defined in a functorial way, it actually defines a
map of $F$-algebraic groups
\[
\rho_\star\colon\gPGO^+(n)\to\gPGO^+(n).
\]

\begin{thm}
  \label{thm:defcorresp}
  The map $\rho_\star$ is an outer automorphism of order~$3$ of
  $\gPGO^+(n)$
  and $\rho_\star^{-1} =\rho_{\star^{\op}}$.
\end{thm}

\begin{proof}  
  Since commutation of $D_\star^+(f,g,h)$ implies that $D_\star^+(g,h,f)$ and
  $D_\star^+(h,f,g)$ commute, we have
  \[
  \rho_\star[g]=[h]\qquad\text{and}\qquad
  \rho_\star[h] = [f],
  \]
  hence $\rho_\star^3=\Id$. Now, let $f_1$, $f_2\in\GO^+(n)$ and
  assume $g_1$, $g_2$, $h_1$, $h_2\in\GO^+(n)$ are such that
  $D_\star^+(f_1,g_1,h_1)$ and $D_\star^+(f_2,g_2,h_2)$ commute. By
  Lemma~\ref{lem:3} we may find scalars $\lambda_1$, $\lambda_2\in
  F^\times$ such that for all $x$, $y\in S$
  \[
  \lambda_1f_1(x\star y)=g_1(x)\star h_1(y)\qquad\text{and}\qquad
  \lambda_2f_2(x\star y)=g_2(x)\star h_2(y),
  \]
  hence
  \[
  \lambda_1\lambda_2f_1f_2(x\star y) = \lambda_1f_1\bigl(g_2(x)\star
  h_2(y)\bigr) = g_1g_2(x)\star h_1h_2(y).
  \]
  Therefore, $\rho_\star[f_1f_2]=[g_1g_2]$, showing that
  $\rho_\star$ is an automorphism. We refer to \cite[(35.6)]{KMRT} for
  the fact that $\rho_\star$ is an outer algebraic group
  automorphism. (That it is outer also follows from the
  fact that $\rho_\star$ lifts to an automorphism of $\gSpin(n)$ that
  is not the identity on the center, see~\eqref{equ:exactseqspin}).
\end{proof}

The definition of $\rho_\star$ above depends on the isomorphism
$\alpha_0^\star$ through the diagram $D_\star^+(f,g,h)$. We next show
that, conversely, the automorphism $\rho_\star$ determines
$\alpha_0^\star$.

\begin{prop}
  \label{prop:rhoalpha}
  Let $\star$ and $\dia$ be symmetric compositions on $(S,n)$. We have
  $\rho_\star=\rho_\dia$ if and only if $\alpha^\star_0=\alpha^\dia_0$.
\end{prop}

\begin{proof}
  If $\alpha^\star_0=\alpha^\diamond_0$, then the diagrams
  $D_\star^+(f,g,h)$ and $D_\dia^+(f,g,h)$ coincide for all $f$, $g$,
  $h\in\GO^+(n)$, hence $\rho_\star=\rho_\dia$. Conversely, suppose that
  $\rho_\star=\rho_\diamond$. Then $\alpha^\diamond_0 \circ (\alpha^\star_0)^{-1}$
  is an automorphism of $(\End S)\times(\End S)$ that commutes with
  the involution $\ad_{n}\times\ad_{n}$ and that makes the
  following diagram commute for any $f\in\GO^+(n)$
  \begin{equation}
    \label{eq:diag2}
    \xymatrix{
    (\End S)\times(\End S)\ar[rr]^{\alpha^\diamond_0 \circ (\alpha^\star_0)^{-1}}
    \ar[d]_{\Int(\rho_\star[f])\times \Int(\rho^2_\star[f])} && (\End
    S)\times (\End S) \ar[d]^{\Int(\rho_\star[f])\times
      \Int(\rho^2_\star[f])}\\
    (\End S)\times(\End S) \ar[rr]^{\alpha^\diamond_0\circ (\alpha^\star_0)^{-1}} &&
    (\End S)\times(\End S).
    }
  \end{equation}
  Suppose first that $\alpha^\diamond_0\circ (\alpha^\star_0)^{-1}$ exchanges the
  two factors, so there exist automorphisms $\theta_1$, $\theta_2$ of $(\End
  S, \ad_{n})$ such that
  \[
  \alpha^\diamond_0 \circ (\alpha^\star_0)^{-1}(a,b)=\bigl(\theta_1(b),\theta_2(a)\bigr)
  \qquad\text{for all $a$, $b\in\End S$.}
  \]
  Automorphisms of $(\End S,\ad_{n})$ are inner automorphisms
  induced by similarities of $(S,n)$, so we may find $s_1$,
  $s_2\in\GO(n)$ such that $\theta_i=\Int[s_i]$ for $i=1$,
  $2$. By commutativity of diagram~\eqref{eq:diag2}, we have
  \[
  \rho_\star[f]\cdot[s_1]=[s_1]\cdot\rho_\star^2[f] 
  \quad\text{and}\quad
  \rho_\star^2[f]\cdot [s_2]=[s_2]\cdot
  \rho_\star[f]
  \qquad\text{for all $f\in\GO^+(n)$}.
  \]
  Thus, $\Int[s_1]\circ\rho_\star^2=\rho_\star$, hence
  $\Int[s_1]=\rho_\star^2$ and therefore
  $\rho_\star=\Int[s_1^2]$. Since the square of any similarity is a
  proper similarity, it follows that $\rho_\star$ is an inner
  automorphism of $\gPGO^+(n)$, a contradiction. Therefore,
  $\alpha^\diamond_0\circ (\alpha^\star_0)^{-1}$ preserves the two factors, and we
  have
  \[
  \alpha^\diamond_0 \circ (\alpha^\star_0)^{-1}=\Int[s_1]\times\Int[s_2]
  \]
  for some $s_1$, $s_2\in\GO(n)$. Commutativity of
  diagram~\eqref{eq:diag2} yields
  \[
  \rho_\star[f]\cdot [s_1]=[s_1]\cdot \rho_\star[f]
  \quad\text{and}\quad \rho_\star^2[f]\cdot [s_2]=[s_2]\cdot
  \rho_\star^2[f]
  \qquad\text{for all $f\in \GO^+(n)$.}
  \]
  Therefore, $[s_1]$ and $[s_2]$ centralize $\PGO^+(n)$. It follows
  that $s_1$ and $s_2$ are proper similarities, since conjugation by
  an improper similarity is an outer automorphism of
  $\PGO^+(n)$. Since the center 
  of $\PGO^+(n)$ is trivial, it follows that $[s_1]=[s_2]=1$,
  hence $\alpha^\diamond_0 \circ (\alpha^\star_0)^{-1}=\Id$.
\end{proof}

\section{The one-to-one correspondence}

As in the preceding section, we fix a $3$-fold Pfister quadratic space
$(S,n)$ over~$F$.
Our goal is to show that the map
$\star\mapsto \rho_\star$ defines a one-to-one correspondence between
symmetric compositions on $(S,n)$ up to a scalar factor and
trialitarian automorphisms of $\gPGO^+(n)$ defined over $F$. 

\medskip

Recall that given a symmetric composition $\star$ on $(S,n)$, linear
operators $\ell_x^\star$ and $r_x^\star$ on $S$ are defined for any $x\in
S$ by 
\[
\ell_x^\star(y)=x\star y\qquad\text{and}\qquad r_x^\star(y)=y\star x
\qquad\text{for $y\in S$.}
\]
If $x$ is anisotropic, condition~(1) in the definition of symmetric
compositions shows that $\ell_x^\star$ and $r_x^\star$ are
similarities of $(S,n)$ with multiplier $\lambda_\star n(x)$. 
We shall see in
Corollary~\ref{cor:improper} below that $\ell_x^\star$ and $r_x^\star$
are improper similarities for every anisotropic vector $x\in S$. At
this stage, we can at least prove:

\begin{lemma}
  \label{lem:eqimproper}
  The following conditions are equivalent:
  \begin{enumerate}
  \item[(a)]
  there exists an anisotropic vector $x\in S$ such that $\ell_x^\star$
  is an improper similarity;
  \item[(b)]
  there exists an anisotropic vector $x\in S$ such that $r_x^\star$ is
  an improper similarity;
  \item[(c)]
  for every anisotropic vector $x\in S$, the similarity $\ell_x^\star$
  is improper;
  \item[(d)]
  for every anisotropic vector $x\in S$, the similarity $r_x^\star$ is
  improper. 
  \end{enumerate}
\end{lemma}

\begin{proof}
  By Lemma~\ref{lem:2.1} we have for all $x\in S$
  \[
  r_x^\star\circ\ell_x^\star=\lambda_\star n(x)\cdot\Id_S,
  \]
  and $\lambda_\star n(x)\cdot \Id_S$ is a proper similarity. Thus
  $r_x^\star$ and  $\ell_x^\star$ are either both proper or
  improper,  proving
  (a)$\iff$(b) and (c)$\iff$(d). Since (c)~$\Rightarrow$~(a) is clear,
  it only remains to show (a)~$\Rightarrow$~(c). For this, we use a
  homotopy argument. Let $x$, $y\in S$ be anisotropic vectors, and let
  $t$ be an indeterminate over $F$. Consider the vector  $u(t)= x(1-t)+yt\in
  S(t)=S_{F(t)}$. It gives rise to a rational morphism  
  \[
  \phi \colon {\mathbb A^1} \dashrightarrow \gGO(S,n), 
  \]
  induced by $t \mapsto \ell^\star_{u(t)}$,
  of the affine line over $F$ to $ \gGO(S,n)$.
   Note that $\phi$ is defined at $a$ if the vector $u(a)= x(1-a)+ya$ is anisotropic.
  Since $\GO(S,n)$ contains an improper similarity $g$ of $S$,
  the variety of  $\gGO(S,n)$ 
  is the disjoint union of two irreducible subvarieties  $\gGO^+(S,n)$
  and  $g\gGO^+(S,n)$.  
  Since $\mathbb A ^1$ is
  irreducible the image of $\phi$ is contained either in
  $\gGO^+(S,n)$ or in $g\gGO^+(S,n)$. It follows immediately 
  that the similarities $\phi(0) =x$ and $\phi(1) =y$ are both proper
  or both improper.
  \end{proof}

\begin{lemma}
  \label{lem:fond}
  Let $f$, $g\in \GO(n)$ and let $\star$ be a symmetric composition on
  $(S,n)$. The following conditions are equivalent:
  \begin{enumerate}
  \item[(a)]
  the map $\diamond\colon S\times S\to S$ defined by
  \[
  x\diamond y=f(x)\star g(y)\qquad\text{for $x$, $y\in S$}
  \]
  is a symmetric composition on $(S,n)$;
  \item[(b)]
  for all $x$, $y\in S$, 
  \[
  f(x)\star g(y)= \mu(g) g^{-1}\bigl(x\star f(y)\bigr);
  \]
  \item[(c)]
  $f$ and $g$ are proper similarities such that
  \[
  \rho^2_\star[f]\cdot \rho_\star[f]\cdot[f]=1\quad\text{and}\quad
  [g]=\rho_\star^2[f]^{-1}\quad\text{in $\PGO^+(n)$}.
  \]
  \end{enumerate}
  When they hold, the multiplier of $\diamond$ is
  $\lambda_\diamond=\mu(f)\mu(g)\lambda_\star$ and
  \[
  \rho_\diamond=\Int[f^{-1}]\circ\rho_\star=
  \Int[g^{-1}]\circ\rho_\star\circ\Int[f]. 
  \]
  Moreover, assuming \emph{(a)--(c)} hold, the following conditions on
  $[h]\in\PGO^+(n)$ are equivalent:
  \begin{enumerate}
  \item[(i)]
  $\rho_\diamond=
  \Int[h^{-1}]\circ\rho_\star\circ\Int[h]$,
  \item[(ii)]
  $[f] =\rho_\star[h]\inv \cdot [h]  = [h] \cdot \rho_\dia[h]\inv $.
  \end{enumerate}
\end{lemma}

\begin{proof}
  We first show (a)$\iff$(b). Since $f$ and $g$ are similarities and
  $\star$ is a symmetric composition, we have
  \begin{equation}
    \label{eq:multdiam}
    n\bigl(f(x)\star g(y)\bigr)=\lambda_\star \mu(f)\mu(g) n(x) n(y)
    \qquad\text{for all $x$, $y\in S$,}
  \end{equation}
  hence the map $\diamond$ satisfies condition~(1) in the definition
  of symmetric compositions. Therefore, (a) is equivalent to
  \[
  b_n(f(x)\star g(y),z)= b_n\bigl(x,f(y)\star g(z)\bigr)
  \qquad\text{for all $x$, $y$, $z\in S$.}
  \]
  Since $\star$ is a symmetric composition and $g$ is a similarity, we
  may rewrite the right side as
  \[
  b_n\bigl(x\star f(y),g(z)\bigr)=\mu(g) b_n\bigl(g^{-1}\bigl(x\star
  f(y)\bigr), z\bigr).
  \]
  Therefore, (a) is equivalent to
  \[
  b_n(f(x)\star g(y),z)=\mu(g) b_n\bigl(g^{-1}\bigl(x\star
  f(y)\bigr),z\bigr) \qquad\text{for all $x$, $y$, $z\in S$.}
  \]
  Since $n$ is nonsingular, this condition is also equivalent to~(b).

  Now, assume~(b) holds. Fixing an anisotropic vector $x\in S$ and
  considering each side of the equation in~(b) as a function of $y$,
  we have
  \[
  \ell_{f(x)}^\star\circ g=\mu(g) g^{-1}\circ\ell_x^\star\circ f.
  \]
  By Lemma~\ref{lem:eqimproper}, the similarities $\ell_{f(x)}^\star\circ
  g$ and $\mu(g) g^{-1}\circ\ell_x^\star$ are both proper or both
  improper, hence this equation shows that $f$ is proper. Likewise,
  fixing an anisotropic vector $y$ and considering each side of the
  equation in~(b) as a function of $x$, we have
  \[
  r_{g(y)}^\star\circ f= \mu(g)g^{-1}\circ r^\star_{f(y)}.
  \]
  Since by Lemma~\ref{lem:eqimproper}
  the similarities $r_{g(y)}^\star$
  and $r_{f(y)}^\star$ are either both proper or both improper,
  this equation shows that $g$ is proper because $f$ is proper.

  By Lemma~\ref{lem:3}, condition~(b) is equivalent to the
  commutativity of diagram $D_\star^+(g^{-1},f,gf^{-1})$, hence to
  \begin{equation}
    \label{eq:eqc}
    \rho_\star[g^{-1}]=[f]\quad\text{and}\quad \rho_\star^2[g^{-1}]=[gf^{-1}].
  \end{equation}
  Since $\rho_\star^3=\Id$, it follows that
  $[g]=\rho_\star^2[f]^{-1}$ and
  $\rho_\star^2[f]\cdot\rho_\star[f]\cdot[f]=1$. We have thus proved
  (b)~$\Rightarrow$~(c).
  Conversely, we readily deduce~\eqref{eq:eqc} from (c), hence 
  the diagram $D_\star^+(g^{-1},f,gf^{-1})$ commutes. By
  Lemma~\ref{lem:3}, we may find $\lambda$, $\mu$, $\nu\in F^\star$
  such that for all $x$, $y\in S$
  \begin{align}
    \label{eq:g}
    \lambda\, g^{-1}(x\star y) & =f(x)\star gf^{-1}(y),\\
    \label{eq:g1}
    \mu \, f(x\star y) & = gf^{-1}(x)\star g^{-1}(y),\\
    \label{eq:g2}
    \nu\, gf^{-1}(x\star y) & = g^{-1}(x) \star f(y),
  \end{align}
  and
  \[
  \lambda\mu=\mu(gf^{-1}),\quad \mu\nu=\mu(g^{-1}),\quad
  \lambda\nu=\mu(f).
  \]
  These last equations yield $\lambda^2\mu\nu=\mu(g)$. To obtain~(b)
  from~\eqref{eq:g}, it suffices to prove $\lambda=\mu(g)$, wich
  amounts to $\lambda\mu\nu=1$. For this, observe that
  \[
  \lambda\mu\nu\,x\star y = \lambda\mu\nu\,gf^{-1}\circ f\circ
  g^{-1}(x\star y)\qquad\text{for all $x$, $y\in S$.}
  \]
  Compute the right side using successively \eqref{eq:g},
  \eqref{eq:g1}, and \eqref{eq:g2}:
  \begin{align*}
    \lambda\mu\nu\,gf^{-1}\circ f\circ g^{-1}(x\star y) & = \mu\nu\,
    gf^{-1}\circ f\bigl(f(x)\star gf^{-1}(y)\bigr)\\
    & = \nu\,gf^{-1}\bigl(g(x)\star f^{-1}(y)\bigr)\\
    & = x\star y.
  \end{align*}
  Thus, $\lambda\mu\nu\,x\star y = x\star y$ for all $x$, $y\in S$,
  hence $\lambda\mu\nu=1$ and it follows that (c)$\Rightarrow$(b).
  
  Now, assume (a), (b), and (c) hold. The equation
  $\lambda_\diamond=\mu(f)\mu(g)\lambda_\star$ easily follows
  from~\eqref{eq:multdiam}. For $\varphi\in\GO^+(n)$
  Theorem~\ref{thm:4} yields
  $\varphi'$, $\varphi''\in\GO^+(n)$ such that
  \[
  \varphi(x\star
  y)=\varphi'(x)\star\varphi''(y)\qquad\text{for all $x$, $y\in S$,}
  \]
  so $\rho_\star[\varphi]=[\varphi']$. Then
  \[
  \varphi(x\diamond y) =
  \varphi\bigl(f(x)\star g(y)\bigr) = \varphi'f(x)\star \varphi''g(y)
  = f^{-1}\varphi'f(x)\diamond g^{-1}\varphi''g(y),
  \]
  hence $\rho_\diamond[\varphi]=[f^{-1}\varphi' f]$. It follows that
  $\rho_\diamond= \Int[f^{-1}]\circ\rho_\star$. From~(c) it is
  easily derived that $[f]^{-1}=[g^{-1}]\cdot\rho_\star[f]$, hence we
  also have $\rho_\diamond = \Int[g^{-1}]\circ\rho_\star\circ\Int[f]$.

  Finally, (i) holds if and only if $\rho_\dia=\Int\bigl([h]^{-1}\cdot
  \rho_\star[h]\bigr)\circ\rho_\star$. This equation is equivalent
  to~(ii) since $\rho_\dia=\Int[f^{-1}]\circ\rho_\star$ and the center
  of $\gPGO^+(n)$ is trivial.
\end{proof}

\begin{thm}
  \label{thm:oneonecor}
  The assignment $\star\mapsto\rho_\star$ defines a one-to-one
  correspondence between symmetric compositions on $(S,n)$ up to
  scalars and trialitarian automorphisms of $\gPGO^+(n)$ defined over $F$.
\end{thm}

\begin{proof}
  We first show that the map is onto.
  Let $\tau$ be a trialitarian automorphism of $\gPGO^+(n)$ over $F$ and let
  $\star$ be a symmetric composition on $(S,n)$, so $\rho_\star$ also
  is a trialitarian automorphism. In view of the exact
  sequence~\eqref{eq:exactseq}, 
  we have either $\tau \equiv\rho_\star$ or $\tau\equiv\rho_\star^{-1}$
  $\bmod\Int\bigl(\PGO^+(n)\bigr)$. Substituting $\star^\op$ for
  $\star$ if necessary, we may assume
  $\tau\equiv\rho_\star\bmod\Int\bigl(\PGO^+(n)\bigr)$, hence there
  exists $f\in\GO^+(n)$ such that
  \[
  \tau=\Int[f^{-1}]\circ\rho_\star.
  \]
  Since $\tau^3=\rho_\star^3=\Id$, we must have
  $[f]^{-1}\cdot\rho_\star[f]^{-1}\cdot\rho_\star^2[f]^{-1}=1$. Let
  $g\in\GO^+(n)$ be such that
  $[g]=\rho_\star^2[f]^{-1}$. Lemma~\ref{lem:fond} then shows that
  $\tau=\rho_\diamond$ for the symmetric composition $\diamond$
  defined by
  \[
  x\diamond y = f(x)\star g(y)\qquad\text{for $x$, $y\in S$.}
  \]

  Now, let $\star$ and $\dia$ be two symmetric compositions on
  $(S,n)$. Suppose $\dia$ is a scalar multiple of $\star$, say there
  exists $\mu\in F^\times$ such that $x\dia y= \mu x\star y$ for all
  $x$, $y\in S$. Then $\lambda_\dia=\mu^2\lambda_\star$, and for all
  $x\in S$ we have $\ell_x^\dia=\mu\ell_x^\star$ and
  $r_x^\dia=\mu r_x^\star$. Therefore, $\alpha_0^\dia=\alpha_0^\star$,
  hence $\rho_\dia=\rho_\star$ by
  Proposition~\ref{prop:rhoalpha}. Thus, symmetric compositions that
  are scalar multiples of each other define the same trialitarian
  automorphism, and it only remains to show the converse: if
  $\star$ and $\diamond$ are symmetric compositions such that
  $\rho_\star=\rho_\diamond$, then $\star$ and $\diamond$ are
  multiples of each other. By Proposition~\ref{prop:rhoalpha}, the hypothesis
  $\rho_\star=\rho_\diamond$ implies $\alpha^\star_0=\alpha^\diamond_0$,
  hence
  \[
  \lambda_\star^{-1}\ell_x^\star r_y^\star = \lambda_\diamond^{-1}
  \ell_x^\diamond r_y^\diamond \qquad\text{for all $x$, $y\in S$.}
  \]
  Fix an anisotropic vector $x\in S$ and let
  $\varphi=(\ell_x^\diamond)^{-1}\ell_x^\star$, so
  \[
  \lambda_\diamond\lambda_\star^{-1} \varphi r_y^\star=r_y^\diamond
  \qquad\text{for all $y\in S$.}
  \]
  This equation means that
  \begin{equation}
    \label{eq:phix}
    z\diamond y=\lambda_\diamond\lambda_\star^{-1} \varphi(z\star y)
    \qquad\text{for all $y$, $z\in S$.}
  \end{equation}
  The map $\varphi$ is a similarity since $\ell_x^\diamond$ and
  $\ell_x^\star$ are similarities. Its multiplier is
  \[
  \mu(\varphi)=\mu(\ell_x^\diamond)^{-1}\mu(\ell_x^\star)=
  \lambda_\diamond^{-1}\lambda_\star.
  \]
  Suppose first $\varphi$ is
  improper. Theorem~\ref{thm:4} then yields improper
  similarities $\varphi'$, $\varphi''$ such that
  \[
  \varphi(z\star y)=\mu(\varphi) \varphi'(y)\star\varphi''(z)
  \qquad\text{for all $y$, $z\in S$},
  \]
  hence
  \[
  z\diamond y=
  \varphi''(z)\star^\op \varphi'(y) \qquad\text{for all $y$, $z\in
    S$.}
  \]
  Lemma~\ref{lem:fond} yields a contradiction since $\varphi'$ and
  $\varphi''$ are improper. Therefore, $\varphi$ is proper and
  Theorem~\ref{thm:4} yields $\varphi'$,
  $\varphi''\in\GO^+(n)$ such that
  \[
  \varphi(z\star y)=\mu(\varphi) \varphi'(z)\star \varphi''(y)
  \qquad\text{for all $y$, $z\in S$,}
  \]
  hence
  \[
  z\diamond y=
  \varphi'(z)\star\varphi''(y) \qquad\text{for all $y$, $z\in S$.}
  \]
  By Lemma~\ref{lem:fond} it follows that
  $\rho_\diamond=\Int[{\varphi'}^{-1}]\circ\rho_\star$. Since by
  hypothesis $\rho_\star=\rho_\diamond$, this equation implies
  $[\varphi']=1$. But $[\varphi']=\rho_\star[\varphi]$, so also
  $[\varphi]=1$. Equation~\eqref{eq:phix} then shows that $\diamond$
  and $\star$ are scalar multiples of each other.
\end{proof}

\section{Classification of conjugacy classes of trialitarian automorphisms}

We next show that under the one-to-one correspondence of
Theorem~\ref{thm:oneonecor} similarity of symmetric
compositions corresponds to conjugacy of trialitarian automorphisms.

\begin{prop}
  \label{prop:class}
  Let $h\colon \dia \to \star$ be a similarity of symmetric
  compositions on $(S,n)$. Then
  \begin{equation} 
  \label{equ:rhoconj1} 
    \rho_\star=
    \Int[h]\circ\rho_\dia\circ\Int[h]^{-1}.
  \end{equation} 
  Conversely, if~\eqref{equ:rhoconj1} holds for some $h\in\GO(n)$,
  then some scalar multiple of $h$ is a similarity
  $\dia\to\star$.
\end{prop}

\begin{proof}
  Suppose first $h\colon\dia\to\star$ is a similarity of symmetric
  compositions, so
  \begin{equation}
    \label{eq:hsim}
    h(x\dia y)=h(x)\star h(y)\qquad\text{for all $x$, $y\in S$.}
  \end{equation}
  Let $\varphi\in\GO^+(n)$. Theorem~\ref{thm:4} yields $\varphi'$,
  $\varphi''\in\GO^+(n)$ such that
  \[
  \varphi(x\star y)=\varphi'(x)\star\varphi''(y)\qquad\text{for all
    $x$, $y\in S$,}
  \]
  and $\rho_\star[\varphi]=[\varphi']$ by definition. Applying
  $\varphi$ to each side of~\eqref{eq:hsim}, we find
  \[
  \varphi h(x\dia y) = \varphi'h(x)\star \varphi''h(y).
  \]
  The right side is
  \[
  hh^{-1}\varphi'h(x)\star hh^{-1}\varphi''h(y)= h\bigl(
  h^{-1}\varphi'h(x)\dia h^{-1}\varphi''h(y)\bigr),
  \]
  hence
  \[
  h^{-1}\varphi h(x\dia y) = h^{-1}\varphi'h(x)\dia
  h^{-1}\varphi''h(y) \qquad\text{for all $x$, $y\in S$.}
  \]
  Therefore, $\rho_\dia\bigl([h^{-1}\varphi h]\bigr) = [h^{-1}]\cdot
  \rho_\star[\varphi]\cdot [h]$, proving~\eqref{equ:rhoconj1}.

  Conversely, assume \eqref{equ:rhoconj1} holds, and set
  \[
  x\rhd y = h^{-1}\bigl(h(x)\star h(y)\bigr) \qquad\text{for $x$,
    $y\in S$,}
  \]
  so $\rhd$ is a symmetric composition similar to $\star$ under
  $h$. The arguments above show that
  \[
  \rho_\rhd=\Int[h^{-1}]\circ\rho_\star\circ\Int[h] = \rho_\dia,
  \]
  hence Theorem~\ref{thm:oneonecor} implies that $\dia$ is a multiple
  of $\rhd$. If $\lambda\in F^\times$ is such that $x\rhd
  y=\lambda\,x\dia y$ for all $x$, $y\in S$, then $\lambda\,h(x\dia
  y)=h(x)\star h(y)$, hence
  \[
  \lambda^{-1}h(x\dia y)= \lambda^{-1}h(x)\star \lambda^{-1}h(y)
  \qquad\text{for all $x$, $y\in S$.}
  \]
  This equation shows that $\lambda^{-1}h\colon\dia\to\star$ is a
  similarity.
\end{proof}

For the following statement, recall from
Proposition~\ref{prop:normalize} that every symmetric composition is
similar to a normalized symmetric composition, which is unique up to
isomorphism.

\begin{thm} \label{thm:bigclassifrev}
Let $\star$ and $\dia$ be symmetric compositions on $(S,n)$, and let
$\overline\star$, $\overline\dia$ be normalized symmetric
compositions similar to $\star$ and to $\dia$ respectively.
The following conditions are equivalent:
\begin{enumerate}
\item[(a)]
$\star$ and $\dia$ are similar;
\item[(b)]
$\overline\star$ and $\overline\dia$ are isomorphic;
\item[(c)]
$\rho_\star$ and $\rho_\dia$ are conjugate in
$\Aut\bigl(\gPGO^+(n)\bigr)$ over $F$;
\item[(d)]
$\rho_{\overline\star}$ and $\rho_{\overline\dia}$ are conjugate in
$\Aut\bigl(\gPGO^+(n)\bigr)$ over $F$.
\end{enumerate}
\end{thm}

\begin{proof}
  It is clear from the definition of $\overline\star$ and
  $\overline\dia$ that (a)$\iff$(b). Since $\star$ and
  $\overline\star$ are similar, Proposition~\ref{prop:class} shows
  that $\rho_\star$ and $\rho_{\overline\star}$ are conjugate in
  $\Aut\bigl(\gPGO^+(n)\bigr)$ over $F$. Similarly, $\rho_\dia$ and
  $\rho_{\overline\dia}$ are conjugate, hence it is clear that
  (c)$\iff$(d). Proposition~\ref{prop:class} also yields
  (a)$\Rightarrow$(c), so it only remains to show (c)$\Rightarrow$(a).

  Suppose $\phi$ is an automorphism of $\gPGO^+(n)$ defined over $F$
  such that
  \[
  \rho_\star=\phi\circ\rho_\dia\circ\phi^{-1}.
  \]
  Let $f\in\GO(n)$ be an improper similitude. The restriction of
  $\Int[f]$ to $\gPGO^+(n)$ is an outer automorphism whose square is
  inner, hence each coset of $\Aut\bigl(\gPGO^+(n)\bigr)$ modulo
  $\Int\bigl(\gPGO^+(n)\bigr)$ is represented  over $F$ by an element from the
  set
  \[
  \{\Id,\;\rho_\dia,\;\rho_\dia^2,\;\Int[f],\;\Int[f]\circ\rho_\dia,\;
  \Int[f]\circ\rho_\dia^2\}.
  \]
  Since $\phi\circ\rho_\dia\circ\phi^{-1}$ does not change when $\phi$
  is multiplied on the right by $\rho_\dia$ or $\rho_\dia^{-1}$, we
  may assume the coset of $\phi$ is represented by $\Id$ or by
  $\Int[f]$, hence $\phi=\Int[h]$ for some $h\in\GO(n)$. It follows
  from Proposition~\ref{prop:class} that $\star$ and $\dia$ are
  similar. 
\end{proof}

Theorem~\ref{thm:bigclassifrev} shows that the correspondence
$\star\mapsto\rho_\star$ induces a one-to-one correspondence between
similarity classes of symmetric compositions, or isomorphism classes
of normalized symmetric compositions, and conjugacy classes of
trialitarian automorphisms over $F$ in $\Aut\bigl(\gPGO^+(n)\bigr)$.

\begin{cor}
  \label{cor:bigclassifrev}
  \begin{enumerate}
  \item Two symmetric compositions $\star$ and $\dia$ on $(S,n)$ are
    similar or antisimilar if and only if the subgroups generated by
    $\rho_\star$ and $\rho_\dia$ in $\Aut\bigl(\gPGO^+(n)\bigr)$ are
    conjugate.
  \item Two normalized symmetric compositions $\star$ and $\dia$ on
    $(S,n)$ are isomorphic or anti-isomorphic if and only if the
    subgroups generated by $\rho_\star$ and $\rho_\dia$ in
    $\Aut\bigl(\gPGO^+(n)\bigr)$ are conjugate.
  \end{enumerate}
\end{cor}

\begin{proof}
  The subgroups generated by $\rho_\star$ and $\rho_\dia$ are
  conjugate if and only if $\rho_\star$ is conjugate to $\rho_\dia$ or
  to $\rho_\dia^{-1}$. Since $\rho_\dia^{-1}=\rho_{\dia^\op}$ by
  Theorem~\ref{thm:defcorresp}, the corollary follows from
  Theorem~\ref{thm:bigclassifrev}. 
\end{proof}

To complete this section, we discuss fixed points of trialitarian
automorphisms, which will be used to classify trialitarian
automorphisms in Section~\ref{section:classificationsymcomp}. For any
symmetric composition $\star$ on $(S,n)$, we let $\gAut(\star)$ be the
$F$-algebraic group of automorphisms of $\star$, whose group of
$F$-rational points $\Aut(\star)$ consists of the isomorphisms
$f\colon\star\to\star$. Corollary~\ref{cor:autoproper} shows that
$\gAut(\star)\subset \gGO^+(n)$.  

\begin{thm}
\label{thm:fixedpoints}
Let $\star$ be a symmetric composition on $(S,n)$. The canonical map
$\phi\colon\gGO^+(n)\to\gPGO^+(n)$ induces an isomorphism from
$\gAut(\star)$ to the subgroup $\gPGO^+(n)^{\rho_\star}$ of
$\gPGO^+(n)$ fixed under the trialitarian automorphism $\rho_\star$.
\end{thm}

\begin{proof}
If $f\in\Aut(\star)$, then $\rho_\star[f]=[f]$, hence $\phi$ gives
rise to a canonical map
\[
\psi\colon\gAut(\star)\to\gPGO^+(n)^{\rho_\star}.
\]
To prove that $\psi$ is an isomorphism it suffices to show
that $\Ker \psi=1$ and that $\psi$ is a quotient map 
(see \cite[Corollary~15.4]{waterhouse:79}).
The kernel of $\psi$ represents the functor which takes a commutative
$F$-algebra $A$ to $\Ker[\gAut(\star)(A)\to
\gPGO^+(n)^{\rho_\star}(A)]$. Let $a$ be in this kernel. Note that 
$\Ker\psi\subset \Ker\,\phi=\gGm$. Hence $a$ is a homothety $r\Id_{S_A}$ for
some $r\in A^{\times}$. But such a map is in 
$\gAut(\star)(A)$ if and only if $r=1$. Thus $\Ker\,\phi$ represents
the trivial functor and so $\Ker\,\phi=1$.
To see that $\psi$ is a quotient map we need to show that for every commutative
$F$-algebra $A$ and every $g\in \gPGO^+(n)^{\rho_\star}(A)$ there exists a
faithfully flat extension $A\to B$ and an element $f\in \gAut(\star)(B)$ such that
$\psi(f)=g$ (see~\cite[Theorem~15.5]{waterhouse:79}). Since $\phi$ is
a quotient map we can find an $F$-algebra $B$ and an element $f_0\in
\gGO^+(n)(B)$ such that $\phi(f_0)=g$. 
Since $g$ is fixed by $\rho_\star$ there are scalars $\alpha$,
$\beta\in B^\times$ such that
\[
f_0(x\star y) = \alpha f_0(x)\star \beta f_0(y)\qquad\text{for all $x$,
  $y\in S_B$.}
\]
Let $f=\alpha\beta f_0$.
Then $f$ is an automorphism of $\star$ and
$\psi(f)=\phi(f_0)=g$.
\end{proof}

\section{The twist of a symmetric composition}
\label{sec:twisting}

As in the preceding sections, we fix a $3$-fold Pfister quadratic
space $(S,n)$ over $F$. If
$\star$ and $\dia$ are symmetric compositions on $(S,n)$, then the
trialitarian automorphisms $\rho_\star$ and $\rho_\dia$ are related by
inner automorphisms of $\gPGO^+(n)$: in view of the exact sequence
\eqref{eq:exactseq}, we have either
\[
\rho_\star\equiv\rho_\dia \bmod\Int\bigl(\gPGO^+(n)\bigr)
\quad\text{or}\quad \rho_\star\equiv\rho_\dia^{-1}
\bmod\Int\bigl(\gPGO^+(n)\bigr).
\]
We use this observation to relate the symmetric compositions $\star$
and $\dia$.

\begin{proposition}
  \label{prop:oneonecor}
  Let $\star$ and $\diamond$ be two symmetric compositions on
  $(S,n)$. If $\rho_\star\equiv\rho_\diamond\bmod
  \Int\bigl(\gPGO^+(n)\bigr)$, then there exist $f$,
  $g\in\GO^+(n)$ such that 
  \begin{equation}
  \label{eq:diastar}
  x\diamond y=f(x)\star g(y)\qquad\text{for all $x$, $y\in S$.}
  \end{equation}
  If
  $\rho_\star\equiv\rho_\diamond^{-1}\bmod\Int\bigl(\gPGO^+(n)\bigr)$,
  then there exist $f$, $g\in \GO^+(n)$ such that
  \[
  x\diamond y = g(y)\star f(x)\qquad\text{for all $x$, $y\in S$.}
  \]
\end{proposition}

\begin{proof}
  Suppose first $\rho_\dia=\Int[f^{-1}]\circ\rho_\star$ for some
  $f\in\GO^+(n)$. Since $\rho_\dia^3=\Id$, we have
  $\rho_\star^2[f]\cdot\rho_\star[f]\cdot[f]=1$. Let $g\in\GO^+(n)$ be
  such that $[g]=\rho^2_\star[f^{-1}]$, and define
  \[
  x\rhd y=f(x)\star g(y)\qquad\text{for $x$, $y\in S$.}
  \]
  By Lemma~\ref{lem:fond}, $\rhd$ is a symmetric composition on
  $(S,n)$, and
  \[
  \rho_\rhd=\Int[f^{-1}]\circ\rho_\star=\rho_\dia.
  \]
  Therefore, Theorem~\ref{thm:oneonecor} shows that $\rhd$ is a scalar
  multiple of $\dia$. Scaling $g$, we may assume~\eqref{eq:diastar}
  holds.

  If $\rho_\star\equiv\rho_\dia^{-1}\bmod \Int\bigl(\gPGO^+(n)\bigr)$,
  then $\rho_{\star^\op}\equiv\rho_\dia \bmod
  \Int\bigl(\gPGO^+(n)\bigr)$ since $\rho_{\star^\op}=\rho_\star^{-1}$
  by Theorem~\ref{thm:defcorresp}. The first part of the proof yields
  $f$, $g\in\GO^+(n)$ such that
  \[
  x\dia y = f(x)\star^\op g(y) = g(y)\star f(x)\qquad\text{for $x$,
    $y\in S$.}
  \]
\end{proof}

When the symmetric composition $\dia$ is given by \eqref{eq:diastar},
then we must have $[g]=\rho_\star^2[f^{-1}]$, by
Lemma~\ref{lem:fond}. The map $g$ is therefore uniquely determined by
$f$ up to a scalar factor, and we say $\dia$ is a \emph{twist} of
$\star$ through the similarity $f$. Thus, by
Proposition~\ref{prop:oneonecor}, given a symmetric composition $\star$
on $(S,n)$, every symmetric composition on $(S,n)$ is a twist of
$\star$ or $\star^\op$.

Twisting has its  origin in Petersson~\cite{petersson:69},
where the following special case is considered: suppose
$f\colon\star\to\star$ is a similarity (i.e., $f$ is an automorphism
of $\star$), and $f^3=\Id_S$. Since $f$ is an automorphism, we have
$f\in\GO^+(n)$ and $\rho_\star[f]=[f]$: see
Theorem~\ref{thm:fixedpoints}. Since $f^3=\Id_S$, we have
$\rho^2_\star[f]\cdot\rho_\star[f]\cdot[f]=1$, hence we may choose
$g=f^{-1}$ in the discussion above. By Lemma~\ref{lem:fond}, the
product
\begin{equation}
  \label{eq:Petwist}
  x\star_f y = f(x)\star f^{-1}(y) \qquad\text{for $x$, $y\in S$}
\end{equation}
defines a symmetric composition.

The idea of twisting was further developed in \cite{elduque:2000a}. 

\begin{cor}
  \label{cor:improper}
  For every symmetric composition $\diamond$ on $(S,n)$ and every
  anisotropic vector $x\in S$, the similarities $\ell^\diamond_x$ and
  $r^\diamond_x$ are improper.
\end{cor}

\begin{proof}
  Let $\star$ be a para-octonion composition on $(S,n)$ (see
  Example~\ref{ex:cayley}). 
  The similarities
  $\ell_1^\star$ and $r_1^\star$, where $1$ is the identity of the octonion algebra,
   coincide with the conjugation map, which is an
  improper isometry. Therefore, by Lemma~\ref{lem:eqimproper}, the
  similarities $\ell_x^\star$ and $r_x^\star$ are improper for every
  anisotropic vector $x\in S$. Propositiom~\ref{prop:oneonecor} shows
  that there exist $f$, $g\in\GO^+(n)$ such that either
  \[
  x\diamond y = f(x)\star g(y)\quad\text{for all $x$, $y\in S$}
  \quad\text{or}\quad
  x\diamond y = g(y)\star f(x)\quad\text{for all $x$, $y\in S$.}
  \]
  Therefore, for every anisotropic vector $x\in S$ we have either
  $\ell_x^\diamond=\ell_{f(x)}^\star \circ g$ and
  $r_x^\diamond= r_{g(x)}^\star\circ f$, or
  $\ell_x^\dia=r_{f(x)}^\star\circ g$ and
  $r_x^\dia=\ell_{g(x)}^\star\circ f$. The similarities
  $\ell_x^\diamond$ and 
  $r_x^\diamond$ are therefore improper.
\end{proof}
   
\section{Spin groups and symmetric compositions}

Let $\star$ be a normalized symmetric composition on a $3$-fold
Pfister quadratic space $(S,n)$. We briefly point out in this section
how the results of the preceding sections can be modified to apply to
the group $\gSpin(n)$ instead of $\gPGO^+(n)$.

Recall that the $F$-rational points of the $F$-algebraic group
$\gSpin(n)$ are given by
\[
\Spin(n) = \{ c \in C_0(S,n)^\times \ | \ cSc\inv \subset S \quad
\text{and} \quad c\sigma(c) =1\} 
\]
(see for example \cite[35.C.]{KMRT}). The isomorphism $\alpha^\star_0$
introduced in Proposition~\ref{prop:triasimilituderevone} can be used  
to give a convenient description of $\Spin(n)$: for $c\in\Spin(n)$,
define $f$, $f_1$, $f_2\in\End_FS$ as follows:
\[
f(x)=cxc\inv\text{ for $x\in S$,}\qquad \alpha_0^\star(c)=
(f_1,f_2).
\]
The map $f$ is a proper isometry. Similarly, since $c\sigma(c)=1$ and
$\alpha_0^\star$ is an isomorphism of algebras with involution, $f_1$
and $f_2$ are isometries. We have $C_0[f]=\Int(c)\rvert_{C_0(S,n)}$, hence the
diagram $D_\star^+(f,f_1,f_2)$ commutes. Therefore, $f_1$ and $f_2$
are proper isometries. Let $\Orth^+(n)$ be the group of proper
isometries of $(S,n)$. The following result is shown
in~\cite[(35.7)]{KMRT} (assuming $\charac F\neq2$):

\begin{prop}
The map $c\mapsto (f,f_1,f_2)_\star$ defines an isomorphism
\[
\Spin(n)  \isom   \{ ({f},f_1,f_2)_\star \ | \ f_i \in
\Orth^+(n)(F),\, {f}(x \star y) =  f_1(x)\star f_2(y),\ x,y \in S\}. 
\]
Moreover any of the three relations
\[
\begin{array}{lcl}
{f}(x \star y) & =  & f_1(x)\star f_2(y)\\
f_1(x \star y) & =  & {f}_2(x)\star f(y)\\
f_2(x \star y) & =  & f(x)\star {f}_1(y)
\end{array}
\]
implies the two others. 
\end{prop}

Observe that the representation of elements of $\Spin(n)$ as triples of 
elements of $\Orth^+(n)$ depends on the choice of the composition
$\star$, hence the notation 
$({f},f_1,f_2)_\star $.

\medskip

We have an obvious trialitarian automorphism $\widehat\rho_\star$ of
$\gSpin(n)$ defined over $F$ by
\[
\widehat\rho_\star \colon ({f} ,f_1,f_2)_\star  \mapsto (f_1,{f}_2,f)_\star.
\]

The center  $\gZ$ of $\gSpin(n)$ is the scheme
$ \gmu_2^2$ with $\gmu_2(F) = \pm 1$. Viewing  $\gZ(F)$ 
as kernel of the multiplication map 
\[
\gmu_2^3(F) \to \gmu_2(F), \quad 
(\varepsilon_1, \varepsilon_2,\varepsilon_3)\mapsto
\varepsilon_1\varepsilon_2\varepsilon_3,
\]
the group scheme $\gZ$ admits a natural $\Sym_3$-action.
The exact sequence of schemes
\begin{equation}
\label{equ:exactseqspin}
1 \to \gZ \to \gSpin(n) \to \gPGO^+(n) \to 1,
\end{equation}
where the morphism $\gSpin(n) \to \gPGO^+(n)$ is induced by 
$(f,f_1,f_2)_\star \mapsto [f]$,
is  $\Sym_3$-equivariant  (see \cite[(35.13)]{KMRT}).
Thus the trialitarian action on $\gSpin(n)$
is a lift of the trialitarian action on $\gPGO^+(n)$.
 
 \medskip
 
Most of the results on trialitarian actions on $\gPGO^+(n)$ hold for
$\gSpin(n)$. 
Moreover, in a similar way we have a natural action of $\Sym_3$ on
$\gSpin(n)$. Let $\gH$ be the semidirect product of $\gSpin(n)$ and
$\Sym_3$. 

\begin{thm} \label{thm:bigclassifspin}
Two normalized symmetric compositions  $\star$ and $\dia$ on $(S,n)$
are isomorphic if and only if the trialitarian automorphisms
$\widehat\rho_\star$ and $\widehat\rho_\dia$ are conjugate 
over $F$ in $\gH$. 
\end{thm}

One of the main steps in the proof for $\gPGO^+(n)$
was to show that any trialitarian automorphism is induced by a 
symmetric composition. We describe this step for $\gSpin(n)$.

\begin{prop}
  \label{prop:ontospin}
  Every trialitarian automorphism of $\gSpin(n)$ over $F$ in $\gH$ has the form
  $\widehat\rho_\dia$ for some normalized symmetric composition~$\dia$.
\end{prop}

\begin{proof}
  We show how to modify the proof of the corresponding statement  in Theorem~\ref{thm:oneonecor}
  for
  $\gPGO^+(n)$.
  Let $\tau$ be a trialitarian automorphism of $\gSpin(n)$ over $F$, let
  $\star$ be a symmetric composition on $(S,n)$ and let
  $\widehat\rho_\star$ be the 
  associated trialitarian automorphism. In view of the exact
  sequence~\eqref{eq:exactseq}, 
  we have either $\tau \equiv\widehat\rho_\star$ or
  $\tau\equiv\widehat\rho_\star^{-1}$ 
  $\bmod\,\Spin(n)$. Substituting $\star^\op$ for
  $\star$ if necessary, we may assume
  $\tau\equiv\widehat\rho_\star\bmod\Spin(n)$, hence there
  exists $ (h, h_1,h_2)_\star  \in\Spin(n)$ such that
  \[
  \tau=\Int\big((h,h_1,h_2)_\star^{-1}\big)\circ\widehat\rho_\star,
  \]
  which means that for all $(f,f_1,f_2)_\star\in\Spin(n)$,
  \[
  \tau\bigl((f,f_1,f_2)_\star\bigr) = (h\inv f_1h, h_1\inv f_2 h_1,
  h_2\inv f h_2)_\star.
  \]
  It follows from $\tau^3=\widehat\rho_\star^3=\Id$ that 
  $h_2h_1 h =\Id_S$. Let
  \[
   x \dia y = h(x) \star h_2\inv(y) \quad \text{for $x$, $y \in S$}.
   \]
 One deduces easily from  $h_2h_1 h =\Id_S$ that $(x\dia y)\dia x =
 x\dia(y \dia x)= n(x) y$ 
 for  $x$, $y \in S$. Thus $\dia $ is a normalized symmetric
 composition on $(S,n)$. 
 For $(f,f_1,f_2)_\star \in \Spin(n)$ and $x$, $y\in S$ we have
 \begin{equation} \label{equ:diatria}
 \begin{array}{lcl}
 f(x \dia y) & = & f\big( h(x) \star h_2\inv(y)\big) \\
 & = & f_1h(x) \star f_2h_2\inv(y) \\
 & = & h\inv f_1h(x) \dia h_2f_2h_2\inv(y)
 \end{array}
 \end{equation}
  Let 
  \[
  \Spin(n)  =   \{ ({f},f_{1}',f'_{2})_\dia \ | \ f,\,f'_1,\,f'_2\in
  \Orth^+(n),\, {f}(x \dia y) =  f_1'(x)\star f_2'(y),\ x,y \in S\}  
  \]
  be the presentation of $\Spin(n)$ using the symmetric composition $\dia$.
  It follows from~\eqref{equ:diatria} that
  \begin{equation*} 
 f_1' = h\inv f_1h \quad \text{and} \quad f_2' = h_2\inv f_2 h_2,
 \end{equation*}
 so that the passage from the presentation of $\Spin(n)$ using
 $\star$ to the presentation of 
$ \Spin(n)$ using $\dia$ is described as
 \begin{equation} \label{equ:passage}
 (f,f_1,f_2)_\star  \mapsto (f, h\inv f_1h, h_2 f_2h_2\inv)_\dia.
 \end{equation}
 Since 
 $
 \widehat\rho_\dia(f,f'_1,f'_2)_\dia = (f'_1,f'_2,f)_\dia,
 $
 we get, using~\eqref{equ:passage},
 \[
  \widehat\rho_\dia\big((f,f_1,f_2)_\star\big) = 
  (h\inv f_1 h, hh_2 f_2 h_2\inv h\inv, h_2\inv fh_2)_\star
  \]
  so that the relation  $h_2h_1 h =\Id_S$ implies $\tau =  \widehat\rho_\dia$.
\end{proof}

\section{Classification of symmetric compositions}
\label{section:classificationsymcomp}

It follows from Theorem~\ref{thm:bigclassifrev} that the classification
of symmetric compositions up to similarity, the classification of
\emph{normalized} symmetric compositions up to isomorphism, 
and the classification of trialitarian automorphisms up to conjugation
are essentially equivalent. In this section we recall the
classification of normalized
symmetric compositions of dimension~$8$ over arbitrary fields. 

We first consider symmetric compositions over algebraically closed
fields. The norm form $n$ is then hyperbolic. Octonion algebras over
such fields are split and it is convenient to choose as a model the
Zorn algebra, which is defined in arbitrary characteristic.

\subsection{The Zorn algebra}
We denote by $\bullet$  the usual scalar product on
 $F^3=F\times F\times F$,  and by $\times$ the vector product:
for $a=(a_1,a_2,a_3)$  and $b=(b_1,b_2,b_3) \in F^3$,
\[
a\bullet b = a_1b_1+a_2b_2+a_3b_3\qquad\text{and}\qquad
a\times b= (a_2b_3-a_3b_2, a_3b_1-a_1b_3, a_1b_2-a_2b_1).
\]
The Zorn algebra is the set of matrices
\begin{equation*}
\mathfrak{Z} =
\left\{\left.
  \begin{pmatrix}
    \alpha & a\\ b&\beta
  \end{pmatrix}
\right|
\alpha,\beta\in F,\;a,b\in F^3\right\}
\end{equation*}
with the product
\[
\begin{pmatrix}
  \alpha & a\\b&\beta
\end{pmatrix}
\cdot
\begin{pmatrix}
  \gamma&c\\d&\delta
\end{pmatrix}
=
\begin{pmatrix}
  \alpha\gamma+a\bullet d&\alpha c+\delta a -b\times d\\
  \gamma b+\beta d+a\times c& \beta\delta + b\bullet c
\end{pmatrix},
\]
the norm
\[
n
\begin{pmatrix}
  \alpha&a\\b&\beta
\end{pmatrix}
=\alpha\beta - a\bullet b,
\]
and the conjugation
\[
\overline{
  \begin{pmatrix}
    \alpha&a\\b&\beta
  \end{pmatrix}
}=
\begin{pmatrix}
  \beta&-a\\-b&\alpha
\end{pmatrix},
\]
which is such that  $\xi\cdot\overline{\xi} = \overline{\xi}\cdot\xi =
n(\xi)$ for all $\xi\in\mathfrak{Z}$ (see \cite[p. 144]{zorn:30}). 

\medskip
\ 
In view of Example~\ref{ex:cayley}
 the product
\begin{equation*}
\begin{pmatrix}
  \alpha&a\\b&\beta
\end{pmatrix}
*
\begin{pmatrix}
  \gamma&c\\d&\delta
\end{pmatrix}
=
\overline{
\begin{pmatrix}
  \alpha&a\\b&\beta
\end{pmatrix}
}\cdot\overline{
\begin{pmatrix}
  \gamma&c\\d&\delta
\end{pmatrix}
}
=
\begin{pmatrix}
  \beta\delta+a\bullet d& -\beta c-\gamma a-b\times d\\
  -\delta b-\alpha d+a\times c&\alpha\gamma+b\bullet c
\end{pmatrix},
\end{equation*} 
defines a symmetric composition on the quadratic space
$(\mathfrak{Z},n)$. We call $\star$ the \emph{para-Zorn composition}. 

\medskip



We use the technique of Section~\ref{sec:twisting} to twist the
para-Zorn composition. Let $p\colon F^3\to F^3$ be the map
$(x_1,x_2,x_3) \mapsto (x_3,x_1,x_2)$, and let
\[
\pi\colon\mathfrak{Z}\to\mathfrak{Z},\qquad
\begin{pmatrix}
  \alpha&a\\b&\beta
\end{pmatrix}=
\begin{pmatrix}
  \alpha&p(a)\\ p(b)&\beta
\end{pmatrix}.
\]
The map $\pi$ is an automorphism of $\star$ of order~$3$, so we may
consider the twisted composition $\star_\pi$ as in~\eqref{eq:Petwist},
which we call the \emph{split Petersson symmetric composition}. (The
algebra $(\mathfrak Z,\star_\pi)$ is also known as the \emph{split
  pseudo-octonion algebra}.)

\begin{thm} \label{thm:petersson}
  Over an algebraically closed field, there are exactly two
  symmetric compositions up to isomorphism: the para-Zorn
  composition and the split Petersson composition.
\end{thm}

\begin{proof}
  The claim is a result of Petersson \cite[Satz 2.7]{petersson:69}
  if the field has characteristic different from $2$ and $3$, and is due to
  Elduque-P\'erez \cite{EP:96} in arbitrary characteristic.   
\end{proof}

\medskip

Over arbitrary fields we consider  two kinds
of symmetric compositions, which we call type I and type II.
Symmetric compositions  of \emph{type} I are
forms of the para-Zorn composition, i.e., $\star$ is of \emph{type} I
if and only if $\star$ is isomorphic to the para-Zorn composition
after scalar extension to an algebraic closure. Similarly, we say that
$\star$ is of \emph{type} II if it is a form of the split
Petersson composition. Thus, Theorem~\ref{thm:petersson}
shows that every symmetric composition is either of type~I or of type~II.

\subsection{Type I}
The classification of symmetric compositions of type~I is particularly
simple.

\begin{thm}
  \label{thm:TypeI}
  Let $(S,n)$ be a $3$-fold Pfister quadratic space over an arbitrary
  field $F$. Up to isomorphism, there is a unique normalized symmetric
  composition of type~I on $(S,n)$, given by the para-octonion composition.
\end{thm}

\begin{proof}
  Forms of the split para-octonion algebra are para-octonion algebras
  (see for example \cite[\S 34 A]{KMRT}), hence symmetric compositions
  of type~I are para-octonion compositions as defined in
  Example~\ref{ex:cayley}. Isomorphisms of octonion algebras are
  isomorphisms of the corresponding para-octonion algebras, and
  octonion algebras with isometric norms are isomorphic (see for
  example \cite[(33.19)]{KMRT}).
\end{proof}

The automorphism group of a symmetric composition of type~I is the
automorphism group of the corresponding octonion algebra; it is a
simple algebraic group of type~$\mathrm{G}_2$. Accordingly, symmetric
compositions of type~I are also called \emph{symmetric compositions of
  type~$\mathrm{G}_2$}. 

\subsection{Type II}

The classification of symmetric compositions of type~II has a
completely different flavor in characteristic~$3$. We first discuss
the case where the characteristic is different from~$3$, and
distinguish two subcases, depending on whether the base field contains
a primitive cube root of unity or not.

Suppose first $F$ is a field of characteristic different from~$3$
containing a primitive cube root of unity $\omega$. Let $A$ be a
central simple $F$-algebra of degree~$3$. For the reduced
characteristic polynomial of $a\in A$, we use the notation
\[
X^3-\Trd(a)X^2+\Srd(a)X-\Nrd(a)\,1,
\]
so $\Trd$ is the reduced trace map on $A$, $\Srd$ is the reduced
quadratic trace map, and $\Nrd$ is the reduced norm. Let $A^0\subset
A$ be the kernel of $\Trd$. We define a multiplication $\star$ on
$A^0$ by
\begin{equation}
  \label{eq:Okubo}
  x\star\, y= \frac{yx -\omega xy}{1-\omega}- \textstyle{\frac{1}{3}}\Trd(xy)\,1
\end{equation}
and a quadratic form $n$ by
\begin{equation}
  \label{eq:normOkubo}
  n(x)=-{\textstyle\frac13}\Srd(x).
\end{equation}

The following result can be found for instance in \cite[(34.19),
(34.25)]{KMRT}: 

\begin{prop}
  \label{prop:Okubo1}
  The quadratic space $(A^0,n)$ is hyperbolic, and $\star$ is a
  normalized symmetric composition on $(A^0,n)$.
\end{prop}

When the algebra $A$ is split, the composition $\star$ is isomorphic
to the split Petersson composition.
Symmetric compositions as in Proposition~\ref{prop:Okubo1} are called
\emph{Okubo compositions}. The automorphism group of an Okubo
composition associated to a central simple algebra $A$ is
$\gPGL_1(A)$. Therefore, these compositions are also called
\emph{symmetric compositions of type~${}^1\!\mathrm{A}_2$}.

\medskip

Suppose next $F$ is a field of characteristic different from~$3$ that
does not contain a primitive cube root of unity. Let $\omega$ be a
primitive cube root of unity in some separable closure of $F$, and let
$K=F(\omega)$, a separable quadratic extension of $F$. Let $B$ be a
central simple $K$-algebra of degree~$3$ with a unitary involution
$\tau$ leaving $F$ fixed. We let $\SSym(B,\tau)^0$ denote the
$F$-vector space of $\tau$-symmetric elements of reduced trace
zero. Formula~\eqref{eq:Okubo} defines a multiplication on
$\SSym(B,\tau)^0$, and the form $n$ of \eqref{eq:normOkubo} is a
quadratic form on $\SSym(B,\tau)^0$.

The following result is proved in \cite[(34.35)]{KMRT}:

\begin{prop}
  \label{prop:Okubo2}
  The quadratic space $(\SSym(B,\tau)^0,n)$ is a $3$-fold Pfister
  quadratic space that becomes hyperbolic over $K$, and $\star$ is a
  normalized symmetric composition on this space.
\end{prop}

Normalized symmetric compositions as in Proposition~\ref{prop:Okubo2}
are also called Okubo compositions. The automorphism group of an Okubo
composition associated to a unitary involution $\tau$ on a central
simple $K$-algebra $B$ is $\gPGU(B,\tau)$. Therefore, these
compositions are also called \emph{symmetric compositions of
  type~${}^2\!\mathrm{A}_2$}.

\begin{theorem} \label{thm:typeAtwo}
  Let $(S,n)$ be a $3$-fold Pfister quadratic space over a field $F$
  of characteristic  different from~$3$. Let $\omega$ be a primitive
  cube root of unity in a separable closure of $F$.
    \begin{enumerate}
     \item
        If $\omega\in F$, then every
        normalized symmetric composition of type~II on $(S,n)$ is
        isomorphic to the Okubo composition associated to some central
        simple $F$-algebra of degree~$3$, uniquely determined up to
        isomorphism. Such compositions exist if and only if $n$ is
        hyperbolic. 
    \item
        If $\omega\notin F$, let $K=F(\omega)$.
        Every normalized symmetric composition of type~II on $(S,n)$
        is isomorphic to the Okubo composition associated to some
        central simple $K$-algebra of degree~$3$ with unitary
        involution, uniquely determined up to isomorphism. Such
        compositions exist if and only if $n$ is split by $K$.
    \end{enumerate}
  \end{theorem}

\begin{proof}
    See Elduque-Myung~\cite[p.~2487]{EM93} or 
    \cite[(34.37)]{KMRT}.
\end{proof}

\medskip

Now, suppose the characteristic of $F$ is~$3$.
Normalized symmetric compositions of type~II over $F$ are extensively
discussed in \cite{EP:96}, 
\cite{elduque:97}, \cite{elduque:99} and
\cite{elduque:2000a}.  They can be viewed as a specialization of symmetric
compositions of type II over fields containing a primitive cube root
of unity. To introduce them we first consider symbol algebras of degree~$3$. 
Any central simple algebra $A$ of degree $3$
over a field containing a cubic primitive root of unity $\om$ is
a symbol algebra, i.e., $A$ has generators $x$, $y$ such that $ x^3 =
a$, $y ^3 =
b$ and $yx = \om xy $ for some $a$, $b \in F^\times$.  We consider the
Okubo composition $\star$ associated with
$A$.  The following basis of $A^0$:
\begin{equation}\label{equ:basisalgebra}
  \begin{array}{llllllll}
    e_1 = x & e_2 = y & e_3 = \om^2xy & e_4 = \om xy \inv \\
    f_1 = x\inv & f_2 = y\inv & f_3 = \om^2x\inv y\inv & f_4 = \om x\inv y
  \end{array}
\end{equation}
is hyperbolic for the form $n$ and the multiplication table with
respect to this basis of the product $\star$ is given by\\[-2ex]
\begin{equation*} 
{\begin{array}{c|cc|cc|cc|cc}
      &          e_1   &    f_1 &        e_2 &    f_2  &     e_3 &     f_3 &     e_4  & f_4  \\
      \hline
      e_1   &  a f_1 &   0                   & 0 & -e_4  &  0 &  -f_2   & -af_3   & 0   \\
      f_1    &  0       &    a\inv e_1  & - f_4 &  0    &  -e_2 &  0  & 0  & -a\inv e_3 \\
      \hline
      e_2   &  -e_3   &    0      &   b f_2& 0         & -b e_4 &    0       &   -e_1  & 0           \\
      f_2   &   0        &    -f_3    &   0 &  b\inv e_2 &0   & - b\inv f_4& 0 &  -f_1      \\
      \hline
      e_3   &  -a f_4   &    0          &   0 & -e_1 & abf_3 & 0 & 0 & -  b f_2 \\
      f_3   &  0   & - a\inv  e_4    & -f_1 & 0  &  0 & a\inv b \inv e_3   &  - b\inv e_2 & 0 \\
      \hline
      e_4   &   0   &   - f_2  & 0      &      -b\inv e_3 &  -af_1 &   0&    a\inv b f_4 &  0          \\
      f_4 & -e_2 & 0 & -bf_3 & 0& 0 & - a\inv e_1 &0 & ab\inv e_4
    \end{array} }
\end{equation*}

This multiplication table does not involve
$\om$, hence specialising gives  a symmetric
 composition $\dia_{a,b}$ over arbitrary
fields, also of characteristic $3$.

\begin{prop}
  1) The symmetric composition $\dia_{1,1}$ is isomorphic to the split
  Petersson composition.\\
  2) Let $\overline F$ be an algebraic closure of $F$ and let $c$, $d
  \in \overline F$ be such that $c^3 = a$ and $d^3 =b$.  The symmetric
  compositions $\dia_{a,b}$ and $\dia_{1,1}$ are isomorphic over $F(c,d)$.
  Thus the symmetric composition $\dia_{a,b}$ splits over
  $F(c,d)$.
\end{prop}
\begin{proof}
  1) and 2) follow by comparing multiplication tables.
\end{proof}

\begin{theorem} \label{thm:classsymIII}
  Let $(S,n)$ be a $3$-fold Pfister quadratic space over a field $F$
  of characteristic~$3$. Every normalized symmetric composition of
  type~II on $(S,n)$ is isomorphic to a symmetric composition of the
  form $\dia_{a,b}$ for some  $a$, $b \in F^\times$.
\end{theorem}

\begin{proof}
  The claim follows from~\cite[p. 291]{elduque:99} and a comparison of 
  multiplication tables.
\end{proof}

Conditions for isomorphism $\dia_{a,b} \isom
\dia_{a',b'}$ are discussed 
in~\cite{elduque:99}.

\medskip

The automorphism groups (as group schemes) of the compositions
$\dia_{a,b}$ are not smooth in 
characteristic $3$ and are, as far as we know, not studied in the
literature.  Their groups of rational points are described in
\cite{elduque:99}.
 
Observe that symmetric compositions of type~II in characteristic~$3$
are not necessarily split over separably 
closed fields, in contrast to symmetric compositions of type~I and to
symmetric compositions of type~II in characteristic different from~$3$.


\end{document}